\documentclass[12pt, leqno]{amsart} 

\usepackage[mathcal]{euscript}
\usepackage{amssymb, amsfonts, amsthm, xy}
\usepackage{graphicx}
\usepackage{stackrel}
\usepackage{tikz}
\usepackage{caption}
\usepackage{subcaption}
\usepackage{float}

\usepackage{pdfsync}

\xyoption{all} \CompileMatrices

\begin{document}


\def\ds{\displaystyle}
\def\arr{{\longrightarrow}}
\def\colim{\mathop{\rm colim}\nolimits}

\def\intertitle#1{

\medskip

{\em \noindent #1}

\smallskip
}

\newtheorem{thm}{Theorem} 
\newtheorem{theorem}[thm]{Theorem}
\newtheorem{lemma}[thm]{Lemma}
\newtheorem{corollary}[thm]{Corollary}
\newtheorem{proposition}[thm]{Proposition}
\newtheorem{example}[thm]{Example}

\theoremstyle{definition}
\newtheorem{definition}[thm]{Definition}
\newtheorem{point}[thm]{}
\newtheorem{exercise}[thm]{Exercise}

\makeatletter
\let\c@equation\c@thm
\makeatother
\numberwithin{equation}{section}



\newcommand{\comment}[1]{}


\title{Every countable group is the fundamental group of some compact subspace of $\mathbb{R}^4$}
\author{Adam J. Prze\'zdziecki$^1$}
\address{Warsaw University of Life Sciences - SGGW, Warsaw, Poland}
\email{adamp@mimuw.edu.pl}

\maketitle
\begin{center}
\today
\end{center}

\comment{
\begin{center}
\vspace{-0.7cm}
{\small \it Institute of Mathematics, Warsaw University, Warsaw, Poland \\
Email address: \verb+adamp@mimuw.edu.pl+ } \\
\end{center}
}


\begin{abstract}
For every countable group $G$ we construct a compact path connected subspace $K$ of $\mathbb{R}^4$ such that $\pi_1(K)\cong G$. Our construction is much simpler than the one found recently by Virk.

\vspace{4pt}
  {\noindent\leavevmode\hbox {\it MSC:\ }}
  {55Q05} 
  \vspace{-9pt}

\end{abstract}


\vspace{25pt}

The purpose of this note is to present a shorter proof of the following.

\begin{theorem}[Virk {\cite[Theorem 2.20]{virk}}]\label{theorem-main}
  For every countable group $G$ there exists a compact path connected subspace $K\subseteq\mathbb{R}^4$ such that $\pi_1(K)\cong G$.
\end{theorem}

This answered a problem posed by Pawlikowski \cite{pawlikowski}. The context and motivation was nicely outlined by Virk so we confine ourselves to presenting the proof.

Let $I=[0,1]$ be the unit interval. We construct $K$ inside $I^4$. We describe it by drawing two dimensional sections and projections of components of $K$. If $(x_1,x_2,x_3,x_4)$ parametrize $I^4$ then $(x_i,x_j)$ above every set of figures below indicate which coordinates parametrize the sections shown, the first being always horizontal. The remaining two coordinates are constant.

Let $(0,0,0,0)$ be the basepoint. Define a loop $g_n$ as the boundary of the triangle, as in Figure (A), with vertices $(0,0,0,0)$, $(1,{1\over n},0,0)$ and $(1,{1 \over n+{1\over 2}},0,0)$.

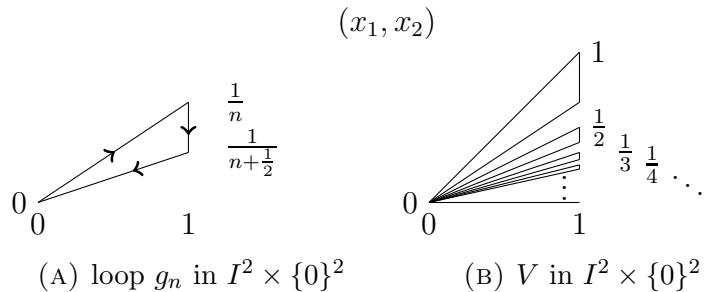
\begin{figure}[H]
\vspace{-15pt}
\caption*{$(x_1,x_2)$}
\vspace{-5pt}

\begin{subfigure}[b]{0.4\textwidth}
\begin{tikzpicture}
  \draw (0,0) node [anchor=north]{$0$};
  \draw (0,0) node [anchor=east]{$0$};
  \draw (2,0) node [anchor=north]{$1$};
  \draw (0,0) -- (2,1.333) node [anchor=west] {$\,\,\,\,\, {1\over n}$};
  \draw (0,0) -- (2,0.666) node [anchor=west] {$\,\,\,\,\, {1\over n+{1\over 2}}$};
  \draw (2,1.333) -- (2,0.666);
  \draw[->, very thick] (2,0.91) -- (2,0.9);
  \draw[->, very thick] (1,0.666) -- (1.02,0.67933);
  \draw[->, very thick] (1.3,0.4329) -- (1.261,0.4199);
\end{tikzpicture}
\caption{loop $g_n$ in $I^2\times\{0\}^2$}
\end{subfigure}
\begin{subfigure}[b]{0.4\textwidth}
\begin{tikzpicture}
  \draw (0,0) node [anchor=north]{$0$};
  \draw (0,0) node [anchor=east]{$0$};
  \draw (2,0) node [anchor=north]{$1$};
  \draw (2,2) node [anchor=west] {$1$} -- (0,0);
  \draw (2,2) -- (2,1.333);
  \draw (0,0) -- (2,1.333);
  \draw (0,0) -- (2,1) node [anchor=west] {$1\over 2$} ;
  \draw (0,0) -- (2,0.8);
  \draw (2,1) -- (2,0.8);
  \draw (0,0) -- (2,0.666) node [anchor=west] {$\,\,\,\,\, {1\over 3}$};
  \draw (0,0) -- (2,0.571);
  \draw (2,0.666) -- (2,0.571);
  \draw (0,0) -- (2,0.5) node [anchor=west] {$\,\,\,\,\,\,\,\,\,\, {1\over 4}$};
  \draw (0,0) -- (2,0.444);
  \draw (2,0.5) -- (2,0.444);
  \draw (2,0.333) node [anchor=west] {$\,\,\,\,\,\,\,\,\,\,\,\,\,\,\,\ddots$};
  \draw (1.8,0.3) node {$\vdots$};
  \draw (0,0) -- (2,0);
\end{tikzpicture}
\caption{$V$ in $I^2\times\{0\}^2$}
\end{subfigure}

\end{figure}

Let
$$
  V=I\times\{0\}^3\cup\bigcup_{n\geq 1}g_n
$$
be the union of the loops $g_n$ for $n\geq 1$ and the interval $I\times\{0\}^3$. Then $V$ is closed and its fundamental group $\pi_1(V)$ is a free group. The elements of its basis are represented by the loops $g_n$ for $n\geq 1$. We use the same symbols $g_n$ to denote the corresponding elements of the fundamental group.

\begin{lemma}\label{lemma-presentation}
  Every countable group $G$ can be presented as
  $$G=\langle g_1,g_2\ldots\mid r_1,r_2,\ldots\rangle$$
  where
  \begin{enumerate}
    \item[(a)] Each relation $r_n$ is of the form $g_ig_jg_k$, $g_ig_j$ or $g_i$ with $i<j<k$.
    \item[(b)] Each generator $g_n$ appears in finitely many relations $r_i$.
  \end{enumerate}
\end{lemma}

\begin{proof}
  Let $(g_n)_{n=1}^\infty$ be a sequence in which every element of $G$ occurs infinitely many times. We add relations as follows. Whenever $g_i=e$ in $G$ we add a relation $g_i$. For every $i$ we choose the least $j$ such that $i<j$ and $g_i=g_j^{-1}$, and we add $g_ig_j$. Finally, for every identity $ab=c$ in $G$ we choose $g_i$, $g_j$ and $g_k$ which, so far, have not appeared in relations of the third type, $i<j<k$, $g_i=a$, $g_j=b$ and $g_k=c^{-1}$. We add $g_ig_jg_k$. These relations reconstruct the multiplication table of $G$ and each generator appears in at most four relations. We label the relations by natural numbers in any way we want.
\end{proof}

Fix a presentation of $G$ as in Lemma \ref{lemma-presentation}.
For every relation $r_n$ we define a subset $R_n\subseteq I^2\times\{{1\over n}\}\times I$. We illustrate the case $r_n=g_ig_j$ by drawing several sections of $R_n$ with planes of the form $I^2\times\{{1\over n}\}\times\{{k\over 4}\}$ where $k=0,1,2,3,4$. The other cases are analogous. Note that the proportions are distorted in order to improve readability of the drawing.

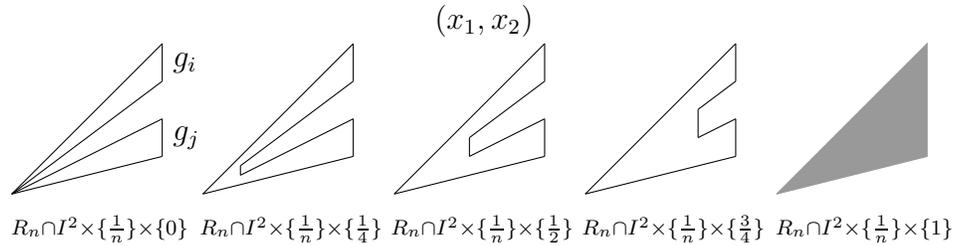
\begin{figure}[H]
\vspace{-15pt}
\caption*{$(x_1,x_2)$}
\vspace{-0pt}

\begin{subfigure}[b]{0.19\textwidth}
\begin{tikzpicture}
  \draw (0,0) -- (2,2);
  \draw (0,0) -- (2,1.5);
  \draw (0,0) -- (2,1);
  \draw (0,0) -- (2,0.5);
  \draw (2,2) -- (2,1.5);
  \draw (2,1) -- (2,0.5);
  \draw (2,1.75) node [anchor=west] {$g_i$};
  \draw (2,0.75) node [anchor=west] {$g_j$};
\end{tikzpicture}
\caption*{$\scriptstyle \! R_n\cap I^2\times\{{1\over n}\}\times\{0\}$}
\end{subfigure}
\begin{subfigure}[b]{0.19\textwidth}
\begin{tikzpicture}
  \draw (0,0) -- (2,2);
  \draw (0.5,0.375) -- (2,1.5);
  \draw (0.5,0.25) -- (2,1);
  \draw (0.5,0.375) -- (0.5,0.25);
  \draw (0,0) -- (2,0.5);
  \draw (2,2) -- (2,1.5);
  \draw (2,1) -- (2,0.5);
\end{tikzpicture}
\caption*{$\scriptstyle \! R_n\cap I^2\times\{{1\over n}\}\times\{{1\over 4}\}$}
\end{subfigure}
\begin{subfigure}[b]{0.19\textwidth}
\begin{tikzpicture}
  \draw (0,0) -- (2,2);
  \draw (1,0.75) -- (2,1.5);
  \draw (1,0.5) -- (2,1);
  \draw (1,0.75) -- (1,0.5);
  \draw (0,0) -- (2,0.5);
  \draw (2,2) -- (2,1.5);
  \draw (2,1) -- (2,0.5);
\end{tikzpicture}
\caption*{$\scriptstyle \! R_n\cap I^2\times\{{1\over n}\}\times\{{1\over 2}\}$}
\end{subfigure}
\begin{subfigure}[b]{0.19\textwidth}
\begin{tikzpicture}
  \draw (0,0) -- (2,2);
  \draw (1.5,1.125) -- (2,1.5);
  \draw (1.5,0.75) -- (2,1);
  \draw (1.5,1.125) -- (1.5,0.75);
  \draw (0,0) -- (2,0.5);
  \draw (2,2) -- (2,1.5);
  \draw (2,1) -- (2,0.5);
\end{tikzpicture}
\caption*{$\scriptstyle \! R_n\cap I^2\times\{{1\over n}\}\times\{{3\over 4}\}$}
\end{subfigure}
\begin{subfigure}[b]{0.19\textwidth}
\begin{tikzpicture}
  \filldraw[black!40!white] (0,0) -- (2,2) -- (2,0.5) -- (0,0);
\end{tikzpicture}
\caption*{$\scriptstyle \! R_n\cap I^2\times\{{1\over n}\}\times\{1\}$}
\end{subfigure}

\end{figure}

We first give a not yet correct anticipation of the construction: Let $W'=V\times I \times\{0\}$ and let $K'=W'\cup\bigcup_{n\geq 1}R_n$. Then $\pi_1(W')\cong\pi_1(V)$ is the free group with basis $\{g_n\}_{n\geq 1}$. Since $W'$ is a neighborhood deformation retract of $K'$ we may use Seifert--van Kampen's Theorem to see that $\pi_1(K')\cong G$. However $K'$ is not closed, and its closure wrecks its fundamental group. We need a bit skinnier replacement for $W'$:

Define $m:\mathbb{N}\arr\mathbb{N}$ by letting $m(n)$ be the least integer such that $g_n$ appears in no $r_m$ for $m>m(n)$. Then $m$ is well defined by Lemma \ref{lemma-presentation}(b).

If we look at each loop $g_n\subseteq V$ separately we see that we do not need to ``fatten'' it to the whole of $g_n\times I\times\{0\}$. In order to attach all the relevant relations, it is enough to take $g_n\times [{1\over m(n)},1]\times\{0\}$. Thus we define
$$
  W'' = I\times \{0\}\times I\times\{0\}\cup\bigcup_{n\geq 1} g_n\times\left[{1\over m(n)},1\right]\times\{0\}
$$
We see that $W''$ is closed. Still $\pi_1(W'')\cong\pi_1(V)$. Our gain is that the intersection $W''\cap I^2\times\{0\}\times I=I\times\{0\}^3$ is now contractible so that if $W=W''\cup I^2\times\{0\}^2$ then $\pi_1(W)\cong\pi_2(W'')\cong\pi_1(V)$ is still the free group on $\{g_n\}_{n\geq 1}$. Define

$$
  M=I^2\times\{0\}\times I\cup W \mbox{ \ \ \ and \ \ \ } K=M\cup\bigcup_{n\geq 1}R_n\,.
$$

Below we show the projections of $W''$, $M$ and $K$ onto $\{0\}^2\times I^2$:

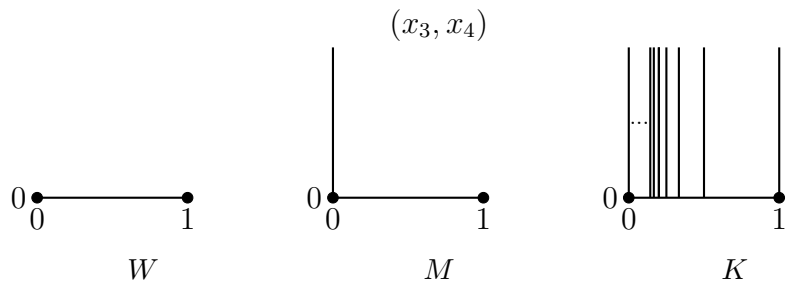
\begin{figure}[H]
\vspace{-15pt}
\caption*{$(x_3,x_4)$}
\vspace{-0pt}

\begin{subfigure}[b]{0.3\textwidth}
\begin{tikzpicture}
  \draw[thick] (0,0) -- (2,0);
  \filldraw (2,0) circle (2pt);
  \filldraw (0,0) circle (2pt);
  \draw (0,0) node [anchor=north] {$0$};
  \draw (0,0) node [anchor=east] {$0$};
  \draw (2,0) node [anchor=north] {$1$};
\end{tikzpicture}
\caption*{$W$}
\end{subfigure}
\begin{subfigure}[b]{0.3\textwidth}
\begin{tikzpicture}
  \draw[thick] (0,0) -- (2,0);
  \draw[thick] (0,0) -- (0,2);
  \filldraw (2,0) circle (2pt);
  \filldraw (0,0) circle (2pt);
  \draw (0,0) node [anchor=north] {$0$};
  \draw (0,0) node [anchor=east] {$0$};
  \draw (2,0) node [anchor=north] {$1$};
\end{tikzpicture}
\caption*{$M$}
\end{subfigure}
\begin{subfigure}[b]{0.3\textwidth}
\begin{tikzpicture}
  \draw[thick] (0,0) -- (2,0);
  \draw[thick] (0,0) -- (0,2);
  \draw[thick] (2,0) -- (2,2);
  \draw[thick] (1,0) -- (1,2);
  \draw[thick] (0.666,0) -- (0.666,2);
  \draw[thick] (0.5,0) -- (0.5,2);
  \draw[thick] (0.4,0) -- (0.4,2);
  \draw[thick] (0.4,0) -- (0.4,2);
  \draw[thick] (0.333,0) -- (0.333,2);
  \draw[thick] (0.2857,0) -- (0.2857,2);
  \draw (0.14,1) node {$\scriptstyle\ldots$};
  \filldraw (2,0) circle (2pt);
  \filldraw (0,0) circle (2pt);
  \draw (0,0) node [anchor=north] {$0$};
  \draw (0,0) node [anchor=east] {$0$};
  \draw (2,0) node [anchor=north] {$1$};
\end{tikzpicture}
\caption*{$K$}
\end{subfigure}

\end{figure}

The left solid dots represent $I^2\times\{0\}^2$. The right ones represent $V$. The left vertical wall in the central figure is the ``compactification wall'', $I^2\times\{0\}\times I$. Clearly $W$ is a deformation retract of $M$, hence $\pi_1(M)\cong\pi_1(V)$ is free on $g_n$'s. In the right figure we add the $R_n$'s. The closedness of $W$ implies the closedness of $K$. Since $W$ is a neighborhood deformation retract of $K$ we use Seifert--van Kampen's Theorem to prove that $\pi_1(K)\cong G$. Theorem \ref{theorem-main} is proved.

We leave it to the reader to notice that the construction above generalizes to the following.
\begin{exercise}
  For every positive integer $n$ and a countable Abelian group $G$ there exists a compact path connected subspace $K\subseteq\mathbb{R}^{n+3}$ such that $\pi_n(K)\cong G$.
\end{exercise}

\end{document}